\documentclass[11pt, letterpaper]{amsart}
\usepackage{epic, eepic, amsfonts, latexsym,amssymb,graphicx,multicol,mathrsfs,color,xypic,amscd, amsxtra, verbatim, paralist, xspace, url, euscript, stmaryrd, newcent}
\usepackage{hyperref}
\input xy
\xyoption{all}

 \newlength{\baseunit}               
 \newcount{\numlines}                
\newenvironment{enumeratea} {\begin{enumerate}[\upshape (a)]} {\end{enumerate}}
\newenvironment{enumeratei} {\begin{enumerate}[\upshape (i)]} {\end{enumerate}}
\newenvironment{enumerate1} {\begin{enumerate}[\upshape (1)]} {\end{enumerate}}

\newtheorem{lem}{Lemma}[section]
\newtheorem{thm}[lem]{Theorem}
\newtheorem{theorem}{Theorem}
\newtheorem{conj}{Conjecture}
\newtheorem{prop}[lem]{Proposition}
\newtheorem{cor}[lem]{Corollary}

\theoremstyle{definition}
\newtheorem{defn}[lem]{Definition}

\newtheorem{example}[lem]{Example}
\newtheorem{counterexample}[lem]{Counterexample}

\theoremstyle{remark}
\newtheorem{remark}[lem]{Remark}

\newcommand{\epf}{\qed \vspace{+10pt}}

 \newcommand\cB{\mathcal{B}}  \newcommand\cF{\mathcal{F}} \newcommand\cG{\mathcal{G}}\newcommand\cI{\mathcal{I}}\newcommand\cK{\mathcal{K}}\newcommand\cN{\mathcal{N}}\newcommand\cO{\mathcal{O}}\newcommand\cP{\mathcal{P}}\newcommand\cU{\mathcal{U}}\newcommand\cV{\mathcal{V}}\newcommand\cW{\mathcal{W}}\newcommand\cX{\mathcal{X}}\newcommand\cY{\mathcal{Y}}\newcommand\cZ{\mathcal{Z}}

\renewcommand\AA{\mathbb{A}}\newcommand\GG{\mathbb{G}}\newcommand\PP{\mathbb{P}}
\newcommand\ZZ{\mathbb{Z}}

    \newcommand\fI{\mathfrak{I}}\newcommand\fU{\mathfrak{U}}

  \newcommand\fm{\mathfrak{m}}      

\newcommand\id{\mathrm{id}}

\newcommand\arr{\ifinner\to\else\longrightarrow\fi}

\newcommand\hookarr{\hookrightarrow}

\renewcommand{\setminus}{\smallsetminus}

\newcommand{\charr}{\operatorname{char}}

\newcommand{\Hom}{\operatorname{Hom}}
\newcommand{\Ext}{\operatorname{Ext}}

\newcommand{\Hilb}{\operatorname{Hilb}}

\newcommand{\sSpec}{\operatorname{\mathcal{S}pec}}

\newcommand{\Sym}{\operatorname{Sym}}
\newcommand{\Aut}{\operatorname{Aut}}

\newcommand{\oh}{\cO}

\newcommand{\Spec}{\operatorname{Spec}}

\newcommand{\tensor} {\otimes}
\newcommand{\wt} {\widetilde}

\newcommand{\Sch}{\operatorname{Sch}}

\newcommand{\Spf}{\operatorname{Spf}}

\newcommand{\iso}{\stackrel{\sim}{\arr}}

\newcommand{\dlim}{{\displaystyle \lim_{\longrightarrow}}\,}

\renewcommand{\bar}{\overline}

\newcommand{\GL}{\operatorname{GL}}
\newcommand{\PGL}{\operatorname{PGL}}

\newcommand{\dual}{\vee}
\renewcommand{\hat}{\widehat}

\begin{document}

\title{On the local quotient structure of Artin stacks}

\author[Alper]{Jarod Alper}



\begin{abstract} 
We show that near closed points with linearly reductive stabilizer, Artin stacks are formally locally quotient stacks by the stabilizer.  We conjecture that the statement holds \'etale locally and we provide some evidence for this conjecture.  In particular, we prove that if the stabilizer of a point is linearly reductive, the stabilizer acts algebraically on a miniversal deformation space, generalizing the results of Pinkham and Rim.  We provide a generalization and stack-theoretic proof of Luna's \'etale slice theorem which shows that GIT quotient stacks are \'etale locally quotients stacks by the stabilizer.
\end{abstract}

\maketitle
\section{Introduction}

This paper is motivated by the question of whether an Artin stack is ``locally'' near a point a quotient stack by the stabilizer at that point.  While this question may appear quite technical in nature, we hope that a positive answer would lead to intrinsic constructions of moduli schemes parameterizing objects with infinite automorphisms (e.g. vector bundles on a curve) without the use of classical geometric invariant theory.

We restrict ourselves to studying Artin stacks $\cX$ over a base $S$ near closed points $\xi \in |\cX|$ with linearly reductive stabilizer.  

 We conjecture that this question has an affirmative answer in the \'etale topology.  Precisely,

\begin{conj}  \label{local_quotient_conj}  If $\cX$ is an Artin stack finitely presented over an algebraic space $S$ and $\xi \in |\cX|$ is a closed point with linearly reductive stabilizer with image $s \in S$, then there exists an \'etale neighborhood $S' \arr S, s' \mapsto s$ and an \'etale, representable morphism $f: [X/G] \arr \cX$ where $G \arr S'$ is a flat and finitely presented group algebraic space acting on an algebraic space $X \arr S'$.  There is a lift of $\xi$ to $x: \Spec k \arr X$ such that the group schemes $\Aut_{\cX(k)}(x)$ and $G \times_{S'} k$ are isomorphic and such that $f$ induces an isomorphism $G_x \arr \Aut_{\cX(k)} (f(x))$.
\end{conj}



For example, if $S = \Spec k$ with $k$ algebraically closed and $x \in \cX(k)$, Conjecture \ref{local_quotient_conj} implies that the stabilizer $G_x$ acts on an algebraic space $X$ of finite type over $k$ fixing some point $\wt x \in X(k)$ and there exists an \'etale, representable morphism $f: [X/G_x] \arr \cX$ mapping $\wt x$ to $x$ and inducing an isomorphism on stabilizer groups.


There are natural variants of Conjecture \ref{local_quotient_conj} that one might hope are true.  One might desire a presentation $[X/G] \arr \cX$ with $X \arr S$ affine and $G \arr S$ linearly reductive; in this case, one would have that \'etale locally on $\cX$, there exists a good moduli space.  One might also like to relax the condition that $G_x$ is linearly reductive to geometrically reductive.  However, some reductivity assumption on the stabilizer seems necessary (see Example \ref{counterexample}).

Conjecture \ref{local_quotient_conj} is known for Artin stacks with quasi-finite diagonals (see Section \ref{conj1_true}).  By a combination of an application of Sumihiro's theorem and Luna's slice argument, this conjecture is true over an algebraically closed field $k$ for global quotient stacks $[X/G]$ where $X$ is a regular scheme separated and of finite type over $k$ and $G$ is a connected algebraic group  (see Section \ref{conj1_true2}).  

However, the conjecture appears to be considerably more difficult for general Artin stacks with non-finite stabilizer group schemes (e.g. $\GG_m^n$, $\PGL_n$, $\GL_n$,...).   To begin with, there is not in general a coarse moduli scheme on which to work \'etale locally.  Second, if $G \arr \Spec k$ is not finite, an action of $G$ on $\Spf A$ for a complete local noetherian $k$-algebra may not lift to an action of $G$ and $\Spec A$ (consider $\GG_m = \Spec k[t]_t$ on $\Spf k[[x]]$ by $x \mapsto tx$) so that for certain deformation functors where one may desire to apply Artin's approximation/algebraization theorems (such as in the proof of \cite[Prop 3.6]{tame}), formal deformations may not be effective.  

While we cannot establish a general \'etale local quotient structure theorem, we establish the conjecture formally locally: 

\begin{theorem} \label{local_quotient_thm}
Let $\cX$ be a locally noetherian Artin stack over a scheme $S$ and $\xi \in |\cX|$ be a closed point with affine linearly reductive stabilizer.  Let $\cG_{\xi} \hookarr \cX$ be the induced closed immersion and $\cX_n$ ($n=1,2, \ldots$) be its nilpotent thickenings.

\begin{enumeratei}
\item If $S = \Spec k$ and there exists a representative $x: \Spec k \arr \cX$ of $\xi$, then there exists affine schemes $U_n$ and actions of $G_x$ on $U_n$ such that $\cX_n \cong [U_n / G_x]$.  If $G_x \arr \Spec k$ is smooth, the schemes $U_n$ are unique up to $G_x$-equivariant isomorphism.
\item Suppose $x: \Spec k \arr \cX$ is a representative of $\xi$ with image $s \in S$ such that $k(s) \hookarr k$ is a finite, separable extension and $G_x \arr \Spec k$ a smooth, affine group scheme.  Fix an \'etale morphism $S' \arr S$ and a point $s' \in S'$ with residue field $k$.  Then there exist affine schemes $U_n$ and linearly reductive smooth group schemes $G_n$ over $S'_n = \Spec \oh_{S',s'}/ \fm_{s'}^{n+1}$ with $G_0 = G_x$ such that $\cX_n \times_S S' \cong [U_n/G_n]$.  The group schemes $G_n \arr S'_n$ are unique and the affine schemes $U_n$ are unique up to $G_n$-equivariant isomorphism.
\end{enumeratei}
\end{theorem}

If, in addition, the stabilizer $G_x$ is smooth, then this theorem implies that $G_x$ acts algebraically on a miniversal deformation space of $\xi$ and this action is unique up to $G_x$-equivariant isomorphism. 

After this paper was written, the author was made aware of similar results by Pinkham and Rim.   In \cite{pinkham}, Pinkham shows that if $\GG_m$ acts on an affine variety $X$ over an algebraically closed field $k$ with an isolated singular point, then the deformation space of $X$ inherits a $\GG_m$-action.  In \cite{rim_equivariant}, Rim showed that for an arbitrary homogeneous category fibered in groupoids, if the stabilizer is a linearly reductive algebraic group, then the stabilizer acts on a miniversal deformation.  

Both Pinkham and Rim follow Schlessinger's approach of building a versal deformation and show inductively that choices can be made equivariantly.   We use an entirely different method.  Following the techniques of \cite{tame}, we use a simple (although technical) deformation theory argument to give a quick proof recovering Rim's result when then category fibered in groupoids is an Artin stack.   Our result is more general in that (1) when the base is a field, we allow for non-reduced stabilizer groups and (2) we can work over any base scheme.  Additionally, Pinkham and Rim appear to give actions on the tangent space and deformation space only by the abstract group of $k$-valued points.  Our methods show immediately that these actions are algebraic.

Luna's \'etale slice theorem implies that GIT quotient stacks are quotient stacks by the stabilizer \'etale locally on the GIT quotient.  More precisely, Luna proved in \cite{luna} that if $G$ is a linearly reductive algebraic group over an algebraically closed field acting on an affine scheme $X$ and if $x \in X$ is a point with closed orbit, then there exists a locally closed, $G_x$-invariant affine $W \subseteq X$ such that the induced morphism on GIT quotients $W//G_x \arr X//G$ is \'etale and such that
$$\xymatrix{
[W/G_x] \ar[r] \ar[d]	& [X/G] \ar[d] \\
W // G_x \ar[r]				& X//G
}$$
is cartesian.  Furthermore, if $x \in X$ is smooth, then $G_x$ acts on a normal space $N_x \subseteq T_x$ to the orbit such that the morphism of GIT quotients $W // G_x \arr N_x // G_x$ is \'etale and
$$\xymatrix{
[W/G_x] \ar[r] \ar[d]	& [N_x/G_x] \ar[d] \\
W // G_x \ar[r]				& N_x//G_x
}$$
is cartesian.  Luna's \'etale slice theorem has had many remarkable applications.

 We prove the following generalization of Luna's \'etale slice theorem.

\begin{theorem} \label{luna_thm} Let $S$ be a noetherian affine scheme.  Let $G \arr S$ be a smooth affine group scheme acting on a scheme $X$ affine and of finite type over $S$ and denote $p: \cX = [X/G] \arr S$.  Suppose $\cX \arr X//G$ is a good moduli space with $X//G \arr S$ of finite type.  Let $f: S \arr X$ be a section such that the stabilizer group scheme $G_f \arr S$ is smooth and the orbit of $f$ is closed (ie.  $o(f) \arr X \times_S T$ is a closed immersion).

\begin{enumeratei}
\item If $X \arr S$ is smooth at points in $f(S)$, there exists a locally closed $G_f$-invariant subscheme $W \hookarr X$ affine over $S$, a normal space to the orbit $N \subset T_{X/S} \times_X S$ with an action of $G_f$, and a $G_f$-equivariant morphism $N \arr W$.  If $\cW = [W/G_f]$ and $\cN = [N /G_f]$, the induced diagram

$${\def\objectstyle{\scriptstyle}
\def\labelstyle{\scriptstyle}
\xymatrix@=20pt{
					& \cW \ar[rd] \ar[ld] \ar[d]\\
\cN  \ar[d]			&  W//G_f \ar[rd] \ar[ld]		& \cX \ar[d]\\
N//G_f				&					& X//G
}}$$
is cartesian with \'etale diagonal arrows.

\item If there is a $G$-equivariant embedding of $X$ into a smooth affine $S$-scheme, there exists $W$ as above such that the diagram
$$\xymatrix{
\cW \ar[r] \ar[d]							& \cX \ar[d] \\
W//G_f \ar[r]							& X//G
}$$
is cartesian with \'etale horizontal arrows.
\end{enumeratei}
\end{theorem}

\begin{remark}
If $S = \Spec k$ with $k$ an algebraically closed field and $G \arr \Spec k$ is a smooth and linearly reductive algebraic group, we recover Luna's slice theorem \cite[p.97]{luna}.  We note that in \cite{luna}, the field $k$ is assumed to have characteristic 0 but the methods of the paper clearly carry over to positive characteristic if $G$ is a smooth and linearly reductive algebraic group.
\end{remark}

\begin{remark} 
The condition that $X$ can be $G$-equivariantly embedded into a smooth affine $S$-scheme is satisfied under very general hypotheses.  If $S$ is regular of dimension 0 or 1, this is well known.  Thomason shows in \cite[Corollary 3.7]{thomason}  that $X$ can be $G$-equivariantly embedded into a vector bundle space $\AA(\cV)$ if:  (1) $S$ is regular with $\dim S \le 2$ and $G \arr S$ has connected fibers, or (2) $G$ is semisimple or split reductive, or (3) $G$ is reductive with isotrivial radical and coradical, or (4) $S$ is normal and $G$ is reductive.
\end{remark}

In particular, over $S = \Spec k$ with algebraically closed, Conjecture \ref{local_quotient_conj} holds for any quotient stack $\cX = [\Spec A/G]$ with $G \arr \Spec k$ a smooth, linearly reductive group scheme around a closed point $\xi \in |\cX|$.

Our statement is slightly more general than Luna's slice theorem.  First, we only require $\cX = [X/G]$ to be a quotient stack admitting a good moduli space with $X$ affine and $G$ an arbitrary smooth, affine group scheme (which is not necessarily linearly reductive).  If $S = \Spec k$ with $\charr(k) = 0$, this is an equivalent formulation since $\GL_n$ is linearly reductive and any quotient stack $[X/G]$ admitting a good moduli space is equivalent to $[\Spec A/\GL_n]$ for some affine scheme $\Spec A$ with a $\GL_n$-action. Second, our version is valid over any noetherian base scheme $S$ with respect to $S$-valued points with closed orbit and smooth, linearly reductive stabilizer.

\subsection*{Acknowledgments}   I thank Dan Abramovich, Johan de Jong, Daniel Greb, Andrew Kresch, Max Lieblich, Martin Olsson, David Smyth, Jason Starr, Ravi Vakil, Fred van der Wyck and Angelo Vistoli for their suggestions.

\section{Background} \label{notation_sec}

We will assume schemes and algebraic spaces to be quasi-separated. An Artin stack, in this paper, will have a quasi-compact and separated diagonal.  We will work over a fixed base scheme $S$.

Recall that if $G \arr S$ is a group scheme acting on an algebraic space $X \arr S$ and $f:T \arr X$ is a $T$-valued point of $X$, then the orbit of $f$, denoted $o(f)$, set-theoretically is the image of $(\sigma \circ (1_G \times f), p_2): G \times_S T \arr X \times_S T$.  We call $G \arr S$ an \emph{fppf group scheme} if $G \arr S$ is a separated, flat, and finitely presented group scheme.  If $G_f \arr T$ is an fppf group scheme, then the orbit has the scheme structure given by
$$\xymatrix{
o(f) \ar[r] \ar[d]	&	X \times_S T \ar[d] \\
BG_f	 \ar[r]		&	[X/G] \times_S T
}$$
If $G_f \arr T$ and $G \arr S$ are smooth group schemes, then $o(f) \arr T$ is smooth.

\subsection{Stabilizer preserving morphisms}
The following definition generalizes the notion of \emph{fixed-point reflecting} morphisms was introduced by Deligne (see \cite[IV.1.8]{knutson}), Koll\'ar  (\cite[Definition 2.12]{kollar_quotients}) and by Keel and Mori (\cite[Definition 2.2]{keel-mori}).  When translated to the language of stacks, the term \emph{stabilizer preserving} seems more appropriate and we will distinguish between related notions.

\begin{defn} Let $f: \cX \arr \cY$ be a morphism of Artin stacks.  We define:
\begin{enumeratei}
\item $f$ is \emph{stabilizer preserving} if the induced $\cX$-morphism $\psi: I_{\cX} \arr I_{\cY} \times_{\cY} \cX$ is an isomorphism.  
\item For $\xi \in |\cX|$, $f$ is \emph{stabilizer preserving at $\xi$} if for a (equivalently any) geometric point $x: \Spec k \arr \cX$ representing $\xi$, the fiber $\psi_x: \Aut_{\cX(k)}(x) \arr \Aut_{\cY(k)} (f(x))$ is an isomorphism of \emph{group schemes} over $k$.
\item $f$ is \emph{pointwise stabilizer preserving} if $f$ is stabilizer preserving at $\xi$ for all $\xi \in |\cX|$.
\end{enumeratei}
\end{defn}

\begin{remark}
Property (i) is requiring that for all $T$-valued points $x: T \arr \cX$, the induced morphism $\Aut_{\cX(T)}(x) \arr \Aut_{\cY(T)} (f(x))$ is an isomorphism of groups.  
\end{remark}

\begin{remark}  One could also consider in (ii) the weaker notion where the morphism $\psi_x$ is only required to be isomorphisms of groups on $k$-valued points.  This property would be equivalent if  $\cX$ and $\cY$ are Deligne-Mumford stacks over an algebraically closed field $k$.
\end{remark}

\begin{remark}  \label{ex_stabilizer}
Any morphism of algebraic spaces is stabilizer preserving.  Both properties are stable under composition and base change.  While a stabilizer preserving morphism is clearly pointwise stabilizer preserving, the converse is not true.  For example, consider the action of $\ZZ_2 \times \ZZ_2 = \langle\sigma, \tau \rangle$ on the affine line with a double origin $X$ over a field $k$ where $\sigma$ acts by inverting the line but keeping both origins fixed and $\tau$ acts by switching the origins.  Then the stabilizer group scheme $S_X \hookarr \ZZ_2 \times \ZZ_2 \times X \arr X$ has a fibers $(1, \tau)$ everywhere except over the origins where fibers are $(1, \sigma)$.  The subgroup $H = \langle 1, \tau \sigma \rangle$ acts freely on $X$ and there is an induced trivial action of $\ZZ_2$ on the non-locally separated line $Y = X / H$.  There is $\ZZ_2$-equivariant morphism $Y \arr \AA^1$ (with the trivial $\ZZ_2$ action on $\AA^1$) which induces a morphism $[Y / \ZZ_2] \arr [\AA^1 / \ZZ_2]$ which is pointwise stabilizer preserving but not stabilizer preserving.  We note that the induced map $[Y / \ZZ_2] \arr \AA^1$ is not a $\ZZ_2$-gerbe even though the fibers are isomorphic to $B\ZZ_2$.  This example arose in discussions with Andrew Kresch.
\end{remark}

It is natural to ask when the property of being pointwise stabilizer preserving is an open condition and what additional hypotheses are necessary to insure that a pointwise stabilizer preserving morphism is stabilizer preserving.  First, we have:

\begin{prop} (\cite[Prop. 3.5]{rydh_quotients}) \label{stab_pres_open}
Let $f: \cX \arr \cY$ be a representable and unramified morphism of Artin stack with $I_{\cY} \arr \cY$ proper.  The locus $\cU \subseteq |\cX|$ over which $f$ is pointwise stabilizer preserving is open and $f|_{\cU}$ is stabilizer preserving.
\end{prop}

\begin{proof}
 The cartesian square
$$\xymatrix{
I_{\cX} \ar@{^(->}[r]^{\psi} \ar[d]						& I_{\cY} \times_{\cY} \cX \ar[d] \\
\cX \ar@{^(->}[r]^{\Delta_{\cX/\cY}}					& \cX \times_{\cY} \cX
}$$
implies that $\psi$ is an open immersion and since the projection $p_2: I_{\cY} \times_{\cY} \cX \arr \cX$ is proper,  the locus $\cU = \cX \setminus p_2 ( I_{\cY} \times_{\cY} \cX \setminus I_{\cX})$ is open.
\end{proof}

\begin{remark}  The proposition is not true if $f$ is ramified: if $f: [\AA^1 / \ZZ_2] \arr [\AA^1 / \ZZ_2]$ where $\ZZ_2$ is acting by the non-trivial involution and trivially, respectively, then $\psi$ is only an isomorphism over the origin.  The proposition also fails without the properness hypothesis:  if $f: [\AA^2 / \GG_m] \arr [\AA^1 / \GG_m]$ where $\GG_m$ is acting by vertical scaling on $\AA^2$ and trivially on $\AA^1$, then $\psi_x$ is only an isomorphism over the $x$-axis. 
\end{remark}

\begin{remark}
The question in general of when a pointwise stabilizer preserving morphism $f: \cX \to \cY$ is stabilizer preserving can be subtle in general.  Even if $\cX$ has finite inertia, the notions are not equivalent. For instance, consider $X = \Spec k[\epsilon]/(\epsilon^2)$ where $k$ is a field with a $\mu_2$ action where $\epsilon$ has degree $1$.  Then $[X/\mu_2] \to B \mu_2$ is pointwise stabilizer preserving but not stabilizer preserving.
\end{remark}

\subsection{Weakly saturated morphisms}
If $f: \cX \arr \cY$ is a morphism of Artin stacks of finite type over a field $k$, the property that closed points map to closed points has several desired consequences (see for instance Theorems \ref{etale_preserving} and \ref{fund_lem}).  However, this does not seem to be the right notion over an arbitrary base scheme as even finite type morphisms of schemes (e.g. $\Spec k(x) \arr \Spec k[x]_{(x)}$) need not send closed points to closed points.   Weakly saturated morphisms will enjoy similar properties.

\begin{defn}  
A morphism $f: \cX \arr \cY$ of Artin stacks over an algebraic space $S$ is \emph{weakly saturated} if for every geometric point $x: \Spec k \arr \cX$ with $x \in |\cX \times_S k|$ closed, the image $f_s(x) \in |\cY \times_S k|$ is closed.  A morphism $f: \cX \arr \cY$ is \emph{universally weakly saturated} if for every morphism of Artin stacks $\cY' \arr \cY$, $\cX \times_{\cY} \cY' \arr \cY'$ is weakly saturated.   
\end{defn}

\begin{remark} Although the above definition seems to depend on the base $S$, it is in fact independent:  if $S \arr S'$ is any morphism of algebraic spaces then $f$ is weakly saturated over $S$ if and only if $f$ is weakly saturated over $S'$.  Any morphism of algebraic spaces is universally weakly saturated.  If $f: \cX \arr \cY$ is a morphism of Artin stacks of \emph{finite type over $S$}, then $f$ is weakly saturated if and only if for every geometric point $s: \Spec k \arr S$, $f_s$ maps closed points to closed points.  If $f: \cX \arr \cY$ is a morphism of Artin stacks of finite type over $\Spec k$, then $f$ is weakly saturated if and only if $f$ maps closed points to closed points.
\end{remark}

\begin{remark}
The notion of weakly saturated is not stable under base change.  Consider the two different open substacks $\cU_1, \cU_2 \subseteq [\PP_1 / \GG_m]$ isomorphic to $[\AA^1 / \GG_m]$ over $\Spec k$.  Then 
$$\xymatrix{
\cU_1 \sqcup \cU_2 \sqcup \Spec k \sqcup \Spec k	\ar[r] \ar[d]		& \cU_1 \sqcup \cU_2 \ar[d] \\
\cU_1 \sqcup \cU_2 \ar[r]									& [\PP_1 / \GG_m]
}$$
is 2-cartesian and the induced morphisms $\Spec k \arr \cU_i$ are open immersions which are not weakly saturated.  
\end{remark}

\begin{remark}
There is a stronger notion of a \emph{saturated} morphism $f: \cX \arr \cY$ requiring for every geometric point $x: \Spec k \arr \cX$ with image $s: \Spec k \arr S$, then $f_s( \overline{\{x\}}) \subseteq |\cX \times_S k|$ is closed.  We hope to explore further the properties of saturated and weakly saturated morphisms as well as develop practical criteria to verify them in future work.
\end{remark}

\begin{remark} \label{saturated_open}
Recall as in \cite[Definition 5.2]{alper_good_arxiv}, 
that if $\phi: \cX \arr Y$ is a good moduli space, an open substack $\cU \subseteq \cX$ is \emph{saturated for $\phi$} if $\phi^{-1} (\phi(\cU)) = \cU$.  In this case, an open immersion $\cU \arr \cX$ is weakly saturated if and only if $\cU$ is saturated for $\phi$.
\end{remark}

\section{Evidence for Conjecture \ref{local_quotient_conj}}

\subsection{Conjecture \ref{local_quotient_conj} is known for stacks with quasi-finite diagonal} \label{conj1_true}

An essential ingredient in the proof of the Keel-Mori theorem (see \cite[Section 4]{keel-mori}) is the existence of \'etale, stabilizer preserving neighborhoods admitting finite, flat covers by schemes.  We note that the existence of flat, quasi-finite presentations was known to Grothendieck (see \cite[Exp V, 7.2]{sga3}).  We find the language of \cite{conrad} more appealing:

\begin{prop} (\cite[Lemma 2.1 and 2.2]{conrad})  \label{prop1} Let $\cX$ be an Artin stack locally of finite presentation over a scheme $S$ with quasi-finite diagonal $\Delta_{\cX/S}$. For any point $\xi \in |\cX|$, there exists a representable, \'etale morphism $f: \cW \arr \cX$ from an Artin stack $\cW$ admitting a finite fppf cover by a separated scheme and point $\omega \in |\cW|$ such that $f$ is stabilizer preserving at $\omega$.  In particular, $\cW$ has finite diagonal over $S$.
\end{prop}

\begin{remark} The stack $\cW$ is constructed as the \'etale locus of the relative Hilbert stack $\Hilb_{V/\cX} \arr \cX$ where $V \arr \cX$ is a quasi-finite, fppf scheme cover.  In fact, the morphism $\cW \arr \cX$ is stabilizer preserving at points $\Spec k \arr \cW$ corresponding to the entire closed substack of $V \times_{\cX} \Spec k$ so that every point $x \in |\cX|$ has some preimage at which $f$ is stabilizer preserving.  If $\cX$ has finite inertia, it follows from Proposition \ref{stab_pres_open} that $f$ is stabilizer preserving.  In fact, as shown in \cite[Remark 2.3]{conrad}, the converse is true:  for $\cX$ as above with a representable, quasi-compact, \'etale, pointwise stabilizer preserving cover $\cW \arr \cX$ such that $\cW$ is separated over $S$ and admits a finite fppf scheme cover, then $\cX$ has finite inertia.  
\end{remark}

We now restate one of the main results from \cite{tame}.

\begin{prop} (\cite[Prop. 3.6]{tame}) \label{prop2}
Let $\cX$ be an Artin stack locally of finite presentation over a scheme $S$ with finite inertia.  Let $\phi: \cX \arr Y$ be its coarse moduli space and let $\xi \in |\cX|$ be a point with linearly reductive stabilizer with image $y \in Y$.  Then there exists an \'etale morphism $U \arr Y$, a point $u$ mapping to $y$, a finite linearly reductive group scheme $G \arr U$ acting on a finite, finitely presented scheme $V \arr U$ and an isomorphism $[V/G] \iso U \times_Y \cX$ of Artin stacks over $U$.  Moreover, it can be arranged that there is a representative of $\xi$ by $x: \Spec k(u) \arr \cX$ such that $G \times_U k(u)$ and $\Aut_{\cX(k(u))}(x)$ are isomorphic as group schemes over $\Spec k(u)$.
\end{prop}

Strictly speaking, the last statement is not in \cite{tame} although their construction yields the statement.


\begin{remark}  In particular, this proposition implies that given any Artin stack $\cX$ locally of finite presentation over a scheme $S$ with finite inertia, the locus of points with linearly reductive stabilizer is open.
\end{remark}

\begin{cor}  Conjecture \ref{local_quotient_conj} is true for Artin stacks $\cX$ locally of finite presentation over $S$ with quasi-finite diagonal.  In fact, \'etale presentations $[X/G]$ can be chosen so that $X$ is affine.
\end{cor}

\begin{proof}  
Given $\xi \in |\cX|$, by Proposition \ref{prop1} there exists an \'etale neighborhood $f: \cW \arr \cX$ stabilizer preserving at some $\omega \in |\cX|$ above $\xi$ such that $\cW$ has finite inertia.  Applying Proposition \ref{prop2} to $\cW$ achieves the result.
\end{proof}

\begin{remark}  In fact, the conjecture is even true for Deligne-Mumford stacks with finite inertia which are not necessarily tame (i.e. have points with non-linearly reductive stabilizer).  This follows easily from (see \cite[Lemma 2.2.3]{abramovich-vistoli} and \cite[Thm 2.12]{olsson_homstacks}).  We wonder if any Artin stack with finite inertia can \'etale locally be written as a quotient stack by the stabilizer.   We note that non-reduced, non-linearly reductive finite fppf group schemes are still geometrically reductive. 
\end{remark}

\subsection{Examples}  Here we list three examples of non-separated Deligne-Mumford stacks and  give \'etale presentations by quotient stacks by the stabilizer verifying Conjecture \ref{local_quotient_conj}.  In these examples, good moduli spaces do not exist Zariski-locally.  We will work over an algebraically closed field $k$ with $\charr k \neq 2$.

\begin{example}  \label{ex1} Let $G \arr \AA^1$ be the group scheme which has fibers isomorphic to $\ZZ_2$ everywhere except over the origin where it is trivial.  The group scheme $G \arr \AA^1$ is not linearly reductive.  The classifying stack $BG$ does not admit a good moduli space Zariski-locally around the origin although there does exist a coarse moduli space.   The cover $f: \AA^1 \arr BG$ satisfies the conclusion of Conjecture \ref{local_quotient_conj}. The morphism $f$ is stabilizer preserving at the origin but nowhere else.  This example shows that one cannot hope to find \'etale charts $[X/G] \arr \cX$ of quotient stacks of linearly reductive group schemes which are pointwise stabilizer preserving everywhere.
\end{example}

\begin{example} \label{ex2} (\emph{$4$ unordered points in $\PP^1$ modulo $\Aut(\PP^1)$})

Consider the quotient stack $\cX = [\PP(V) / \PGL_2]$ where $V$ is the vector space of degree 4 homogeneous polynomials in $x$ and $y$.  Let $\cU \subseteq \cX$ be the open substack consisting of points with finite automorphism group.  Any point in $p \in \cU$ can be written as $xy(x-y)(x-\lambda y)$ for $\lambda \in \PP^1$.  If $\lambda \neq 0,1$ or  $\infty$, the stabilizer is
$$G_p = \bigg\{ \begin{pmatrix} 1 & 0 \\ 0 & 1 \end{pmatrix} , \begin{pmatrix} 0 & \lambda \\ 1 & 0 \end{pmatrix}, \begin{pmatrix} \lambda & - \lambda \\ 1 & - \lambda \end{pmatrix}, \begin{pmatrix} 1 & -\lambda \\ 1 & -1 \end{pmatrix} \bigg\}$$

When $\lambda = 1$ (resp. $\lambda = 0$, $\lambda = \infty$), the only elements of the stabilizer are the identity and $\begin{pmatrix} 0 & 1 \\ 1 & 0 \end{pmatrix}$, (resp. $\begin{pmatrix} 1 & 0 \\ 1 & -1 \end{pmatrix}, \begin{pmatrix} 1 & -1 \\ 0 & -1 \end{pmatrix}$).  Therefore, the stabilizer group scheme of the morphism $\PP^1 \arr \cU$ is a non-finite group scheme which is $\ZZ_2 \times \ZZ_2 \arr \PP_1$ but with two elements removed over each of the fibers over $0,1$ and $\infty$ (so that the generic fiber is $\ZZ_2 \times \ZZ_2$ and the fibers over $0, 1$ and $\infty$ are $\ZZ_2$.

We give an \'etale presentation around $1$.  Let $\ZZ_2$ act on $X=\AA^1 \setminus \{0\}$ via $\lambda \mapsto 1/\lambda$.  The morphism $f: X \arr \PP^4, \lambda \mapsto [xy(x-y)(x-\lambda y)]$ is $\ZZ_2$ invariant where $\ZZ_2$ acts on $\PP^4$ via the inclusion $\ZZ_2 \hookarr \PGL_2$ defined by $-1 \mapsto \begin{pmatrix}0 & 1 \\ 1 & 0 \end{pmatrix}$.  The induced morphism $[X/\ZZ_2] \arr \cX$ is \'etale and stabilizer preserving at $1$.  However, it is not pointwise stabilizer preserving in a neighborhood of 1.  The $j$-invariant $j: \cU \arr \PP^1, [xy(x-y)(x-\lambda y)] \mapsto [(\lambda^2-\lambda+1)^3, \lambda^2(\lambda-1)^2]$ gives a coarse moduli space.  The morphism $j$ is not separated and $j$ is not a good moduli space (i.e. $j_*$ is not exact on quasi-coherent sheaves). 



\end{example}

The following example due to Rydh shows that coarse moduli spaces (or even categorical quotients) may not exist for non-separated Deligne-Mumford stacks.

\begin{example} \label{ex3}
The Keel-Mori theorem states that any Artin stack $\cX \arr S$ where the inertia stack $I_{\cX/S} \arr \cX$ is finite admits a coarse moduli space. The finiteness of inertia hypothesis cannot be weakened to requiring that the diagonal is quasi-finite.  Let $X$ be the non-separated plane attained by gluing two planes $\AA^2 = \Spec k[x,y]$ along the open set $\{x \ne 0\}$.  The action of $\ZZ_2$ on $\Spec k[x,y]_x$ given by $(x,y) \mapsto (x,-y)$ extends to an action of $\ZZ_2$ on $X$ by swapping and flipping the axis (explicitly, if $X = U_1 \cup U_2$, the multiplication is defined by $\ZZ_2 \times U_1 \arr U_2, (x_2,y_2) \mapsto (x_1, -y_1)$ and $\ZZ_2 \times U_2 \arr U_1, (x_1,y_1) \mapsto (x_2, -y_2)$).  Then $\cX = [X / \ZZ_2]$ is a non-separated Deligne-Mumford stack.  There is an isomorphism $\cX = [\AA^2 / G]$ where $G = \AA^1 \sqcup (\AA^1 \setminus \{0\}) \arr \AA^1$ is the group scheme over $\AA^1$ whose fibers are $\ZZ_2$ over the origin where it is trivial and $G$ acts on $\AA^2 = \Spec k[x,y]$ over $\AA^1 = \Spec k[x]$ by the non-trivial involution $y \mapsto -y$ away from the origin.

Rydh shows in \cite[Example 7.15]{rydh_quotients} that this stack does not admit a coarse moduli space.   In fact, there does not even exist an algebraic space $Z$ and a morphism $\phi: \cX \arr Z$ which is universal for maps to \emph{schemes}.  The above statements are also true for any open neighborhood of the origin. 
\end{example}

The following is a counterexample for Conjecture \ref{local_quotient_conj} if the stabilizer is not linearly reductive.

\begin{counterexample} \label{counterexample}
Over a field $k$, let $G \arr \AA^1$ be a group scheme with generic fiber $\GG_m$ and with a $\GG_a$ fiber over the origin.  Explicitly, we can write $G = \Spec k[x,y]_{xy+1} \arr \Spec k[x]$ with the multiplication $G \times_{\AA^1} G \arr G$ defined by $y \mapsto xyy' + y + y'$.  Let $\cX = [\AA^1 / G]$ be the quotient stack over $\Spec k$ and $x: \Spec k \arr \cX$ be the origin.  The stabilizer $G_x = \GG_a$ acts trivially on the tangent space $\bar F_{x}(k[\epsilon])$.  The nilpotent thickening $\cX_1$ cannot be a quotient stack by $\GG_a$ giving a counterexample to Conjecture \ref{local_quotient_conj}  in the case that the stabilizer is not linearly reductive.
\end{counterexample}

\subsection{Conjecture \ref{local_quotient_conj} is known for certain quotient stacks} \label{conj1_true2}

\begin{theorem} 
  Let $\cX$ be an Artin stack over an algebraically closed field $k$.  Suppose $\cX=[X/G]$ is a quotient stack  and $x \in \cX(k)$ has smooth linearly reductive stabilizer. Suppose that one of the following hold:
 \begin{enumerate1}
 \item  $G$ is a connected algebraic group acting on a regular scheme $X$ separated and of finite type over $\Spec k$, or
 \item $G$ is a smooth linearly reductive algebraic group acting on an affine scheme $X$. 
\end{enumerate1}
Then there exists a locally closed $G_x$-invariant affine $W \hookarr X$ with $w \in W$ such that
$$[W/G_x] \arr [X/G]$$
is affine and \'etale.
\end{theorem}

\begin{proof}  Part (2) follows directly from Luna's \'etale slice theorem (see \cite{luna} and Section \ref{luna_sect}).  

For part (1), by applying \cite[Theorem 1 and Lemma 8]{sumihiro1}, there exists an open $G$-invariant $U_1$ containing $x$ and an $G$-equivariant immersion $U_1 \hookarr Y= \PP(V)$ where $V$ is a $G$-representation.  Since the action of $G_x$ on $\Spec \Sym^* V^{\dual}$ fixes the line spanned by $x$, there exists a $G_x$-semi-invariant homogeneous polynomial $f$ with $f(x) \neq 0$.  It follows that $V=Y_f \cap U_1$ is a $G_x$-invariant quasi-affine neighborhood of $x$ with $i: V \hookarr \bar V := (\bar{U_1})_f$ is an open immersion and $\bar V$ is affine.  Let $\pi: \bar V \arr \bar V //G_x$ be the GIT quotient.  Since $\bar V \setminus V$ and $x \in V$ are disjoint $G_x$-invariant closed subschemes, $\pi(\bar V \setminus V)$ and $\pi(x)$ are closed and disjoint.  Let $Z \subseteq (\bar V //G_x) \setminus (\pi(\bar V \setminus V))$ be an affine open subscheme containing $x$.  Then $U=\pi^{-1}(Z)$ is a $G_x$-invariant affine open subscheme containing $x$.

The stabilizer acts naturally on $T_x X$ and there exists a $G_x$-invariant morphism $U \arr T_x X$ which is \'etale since $x \in X$ is regular.  Since $G_x$ is linearly reductive, we may write $T_x X = T_x o(x) \oplus W_1$ for a $G_x$-representation $W_1$.  Define the $G_x$-invariant affine $W \subseteq U$ by the cartesian diagram
$$\xymatrix{
W \ar[r] \ar@{_(->}[d]	& W_1 \ar@{_(->}[d] \\
U \ar[r]		& T_x X = T_x o(x) \oplus W_1
}$$
and let $w \in W$ be the point corresponding to $x$.  

The stabilizer $G_x$ acts on $G \times W$ via $h \cdot (g,w) = (gh^{-1}, h \cdot w)$ for $h \in G_x$ and $(g,w) \in G \times W$.  The quotient $G \times_{G_x} W := (G \times W) /G_x$ is affine.  Since the quotient morphism $G \times W \arr G \times_{G_x} W$ is a $G_x$-torsor it follows that $T_{(g,e)} G \times_{G_x} W = (T_{e} G \oplus T_w W) / T_{e} G_x$ where $T_{e} G_x \subseteq T_{e} G \oplus T_w W$ is induced via the inclusion $G_x \arr G \times W, h \mapsto (h^{-1}, h \cdot w)$.  Therefore, $G \times_{G_x} W \arr X$ is \'etale at $(e, w)$.  Furthermore, $G \times_{G_x} W \arr X$ is affine.  It follows that the induced morphism of stacks  $f: [W/G_x] \arr [X/G]$ is affine and \'etale at $w$.  

Let $\phi: [W/G_x] \arr Y=W//G_x$ be the good moduli space corresponding to the GIT quotient $\pi: W \arr W//G_x$.  If $\cZ \subseteq [W/G_x]$ is the closed locus where $f$ is not \'etale, then $\cZ$ is disjoint to $\{w\}$ and it follows that $\phi(\cZ)$ and $\phi(w)$ are closed and disjoint.  Let $Y' \subseteq Y$ is an open affine containing in $Y \setminus \phi(\cZ)$ containing $\phi(w)$ so that $\phi^{-1}(Y') = [W'//G_x]$
where $W' = \pi^{-1}(Y')$ is a $G_x$-invariant affine.  The morphism $f: [W'/G_x] \arr [X/G]$ satisfies the desired properties.
\end{proof}


\section{Actions on deformations}

\subsection{Setup}

Let $\cX$ be a category fibered in groupoids over $\Sch/S$ with $S = \Spec R$.

For an $R$-algebra $A$, an object $a \in \cX(A)$, and a morphism $A' \arr A$ of $R$-algebras, denote by $F_{\cX, a}(A')$ the category of arrows $a \arr a'$ over $\Spec A \arr \Spec A'$ where a morphism $(a \arr a_1) \arr (a \arr a_2)$ is an arrow $a_1 \arr a_2$ over the identity inducing a commutative diagram
$${\def\objectstyle{\scriptstyle}
\def\labelstyle{\scriptstyle}
\xymatrix@=20pt{
	& a_1 \ar[dd] \\
a \ar[ru] \ar[rd]\\
			& a_2
}}$$
Let $\bar F_{\cX, a}(A')$ be the set of isomorphism classes of $F_{\cX,a}(A')$.  When there is no risk of confusion, we will denote $F_a(A') := F_{\cX, a}(A')$ and $\bar F_a(A') = \bar F_{\cX,a}(A')$.

For an $A$-module $M$, denote by $A[M]$ the $R$-algebra $A \oplus M$ with $M^2 = 0$.  

\begin{defn}
We say that $\cX$ is \emph{S1(b)} (resp. \emph{strongly S1(b)}) if for every surjection $B \arr A$ (resp. any morphism $B \arr A$),  finite $A$-module $M$, and arrow $a \arr b$ over $\Spec A \arr \Spec B$, the canonical map
$$\bar F_b(B[M]) \arr \bar F_a(A[M])$$
is bijective.  
\end{defn}

\begin{remark} We are using the notation from \cite{artin_versal} although we are not assuming that $A$ is reduced.  Recall that there is another condition S1(a) such that when both S1(a) and S1(b) are satisfied (called \emph{semi-homogeneity} by Rim), then there exists a miniversal deformation space (or a hull) by \cite{schlessinger} and \cite{rim_equivariant}.  We are isolating the condition S1(b) and strongly S1(b) to indicate precisely what is necessary for algebraicity of the action of the stabilizer on the tangent space.
\end{remark}

\begin{remark}  Any Artin stack $\cX$ over $S$ satisfies the following homogeneity property: for any surjection of $R$-algebras $C' \arr B'$ with nilpotent kernel, $B \arr B'$ any morphism of $R$-algebras, and $b' \in \cX(B')$, the natural functor
\begin{equation} \label{homogeneous}
\cX_{b'}(C' \times_{B'} B) \arr \cX_{b'}(C') \times \cX_{b'}(B)
\end{equation}
is an equivalence of categories (see \cite[Lemma 1.4.4]{olsson_crystalline}).  In particular, any Artin stack $\cX$ over $S$ is strongly S1(b).
\end{remark} 

\bigskip

It is easy to see that if $\cX$ satisfies S1(b), then for any $R$-algebra $A$, object $a \in \cX(A)$ and finite $A$-module $M$, the set $\bar F_a(A[M])$ inherits an $A$-module structure.  In particular, for $x \in 
\cX(k)$, the tangent space $\bar F_x(k[\epsilon])$ is naturally a $k$-vector space.  For any $k$-vector space, the natural identification $\Hom(k[\epsilon], k[V]) \cong V$ induces a morphism
$$\bar F_x(k[\epsilon]) \tensor_k V \arr \bar F_x(k[V])$$
which is an isomorphism for finite dimensional vector spaces $V$. 

\begin{remark} \label{finite_pres} If $\cX$ is also locally of finite presentation, then this is an isomorphism for any vector space $V$ since if we write $V = \dlim V_i$ with $V_i$ finite dimensional then $\dlim \bar F_x(k[V_i]) \arr \bar F_x(k[V])$ is bijective.  
\end{remark}

\subsection{Actions on tangent spaces} \label{inf_action}

For $a \in \cX(A)$, the abstract group $\Aut_{\cX(A)}(a)$ acts on the $R$-module $\bar F_{a}(A[\epsilon])$ via $A$-module isomorphisms:   if $g \in \Aut_{\cX(A)}(a)$ and $(\alpha: a \arr a') \in \bar F_{a}(A[\epsilon])$, then $g \cdot (a \arr a') = (a \stackrel{g^{-1}} {\arr} a \stackrel{\alpha}{\arr} a')$.  

\begin{remark} \label{translation1} For example, suppose $\cX$ is parameterizing flat families of schemes and $X_0 \arr \Spec A$ is an object in $\cX(A)$.  An element $g \in \Aut(X_0)$ acts on infinitesimal deformations via
$$\left(
\vcenter{ \xymatrix   {
X_0 \ar[d] \ar@{^(->}[r]^i		& X \ar[d]^p \\
\Spec A \ar@{^(->}[r]			& \Spec A[\epsilon]
}}
\right) \stackrel{g}{\longmapsto} \left(
\vcenter{ \xymatrix   {
X_0 \ar[d] \ar@{^(->}[r]^{g^{-1} \circ i}		& X \ar[d]^p \\
\Spec A \ar@{^(->}[r]			& \Spec A[\epsilon]
}}
\right)$$
\end{remark}

\bigskip

If $x \in \cX(k)$ with stabilizer $G_x$, we have shown that there is a homomorphism of abstract groups
$$G_x(k) \arr \GL(\bar F_x(k[\epsilon]))(k)$$
We are interested in determining when this is algebraic (i.e. arising from a morphism of group schemes $G_x \arr \GL(\bar F_x(k[\epsilon]))$.
For any $k$-algebra $A$, let $a \in \cX(A)$ be a pullback of $x$.  Note that there is a canonical identification $\Aut_{\cX(A)}(a) \cong G_x(A)$ which induces a homomorphism  
$$G_x(A) \arr \GL(\bar F_{a}(A[\epsilon]))(A)$$

If $\cX$ is strongly S1(b), then using the isomorphism $A[\epsilon] \times_A k \arr k[A]$, we have a bijection $\bar F_x(k[A]) \arr \bar F_{a}(A[\epsilon])$.  The natural maps induce a commutative diagram of $A$-modules
$$\xymatrix{
\bar F_x (k[\epsilon]) \tensor_k A \ar[r] \ar[rd]	& \bar F_x (k[A]) \ar[d] \\
	&  \bar F_{a} (A[\epsilon])
}$$
If $\cX$ is locally of finite presentation over $S$, by Remark \ref{finite_pres}, the top arrow is bijective so that the diagonal arrow is as well.  Therefore, we have a natural homomorphism of groups
$$G_x(A) \arr \GL(\bar F_x (k[\epsilon]) \tensor_k A)(A) = \GL(\bar F_x(k[\epsilon]))(A)$$
for any $k$-algebra $A$ which induces a morphism of group schemes $G_x \arr \GL(\bar F_x(k[\epsilon]))$.

Therefore, if $\cX \arr S$ is locally of finite presentation and is strongly S1(b), then for $x \in \cX(k)$, the stabilizer $G_x$ acts algebraically on $\bar F_x(k[\epsilon])$.  In particular,

\begin{prop}  If $\cX$ is an Artin stack locally of finite presentation over a scheme $S$ and $x \in \cX(k)$, then the stabilizer $G_x$ acts algebraically on the tangent space $\bar F_x(k[\epsilon])$. \epf
\end{prop}

\begin{remark}
The above proposition is certainly well known, but we are unaware of a rigorous proof in the literature.  We thank Angelo Vistoli for pointing out the simple argument above.

In \cite[Prop. 2.2]{pinkham}, Pinkham states that if $\cX$ is the deformation functor over an algebraically closed field of an affine variety with an isolated singular point with $\GG_m$-action, then the tangent space $T^1$ inherits an algebraic $\GG_m$-action.  However, it appears that he only gives a homomorphism of algebraic groups $\GG_m(k) \arr \GL(T^1)(k)$.  There are certainly group homomorphisms $k^* \arr \GL_n(k)$ which are not algebraic.

In \cite[p. 220-1]{rim_equivariant}, Rim states that if $\cX$ is category fibered in groupoids over the category $\cB$ of local Artin $k$-algebras with residue field $k$ with $\cX(k) = \{x\}$ which is homogeneous in the sense that (\ref{homogeneous}) is an equivalence for a surjection $C' \arr B'$ and any morphism $B \arr B'$ in $\cB$, then $\bar F_x(k[\epsilon])$ inherits a linear representation.   However, he only shows that there is a homomorphism of algebraic groups $G_x(k) \arr \GL(\bar F_x(k[\epsilon]))(k)$.  While it is clear that there are morphisms of groups $G_x(A) \arr \GL(\bar F_x(k[\epsilon]))(A)$ for local Artin $k$-algebras with residue field $k$, it is not clear to us that this gives a morphism of group schemes $G_x \arr \GL(\bar F_x(k[\epsilon]))$ without assuming a stronger homogeneity property.
\end{remark}

\subsection{Actions on deformations}

Let $\cX$ be an Artin stack over $S$ and suppose $G \arr S$ is a group scheme with multiplication $\mu: G \times_S G \arr G$ acting on a scheme $U \arr S$ via $\sigma: G \times_S U \arr U$.  To give a morphism
$$[U/G] \arr \cX$$
is equivalent to giving an object $a \in \cX(U)$ and an arrow $\phi: \sigma^* a \arr p_2^* a$ over the identity satisfying the cocycle condition $p_{23}^* \phi \, \circ \, (\id \times \sigma)^* \phi = (\mu \times \id)^* \phi$.  We say that \emph{$G$ acts on $a \in \cX(U)$} if such data exists.  (In fact, there is an equivalence of categories between $\cX([U/G])$ and the category parameterizing the above data.)

\begin{remark}  \label{translation2}Suppose $\cX$ is the Artin stack over $\Spec \ZZ$ parameterize smooth curves.  Suppose that we are given a smooth family of curves $X \to U$ (i.e. an objective of $\cX$ over $U$). 
If $G$ acts on the scheme $U$, then giving a morphism $[U/G] \arr \cX$ is equivalent to giving an action of $G$ on $X$ compatible with the action on $U$. 
\end{remark}



\subsection{Action of formal deformations}

Let $\fU$ be a noetherian formal scheme over $S$ with ideal of definition $\fI$.  Set $U_n$ to be the scheme $(|\fU|, \oh_{\fU}/\fI^{n+1})$.   If $\cX$ is an category fibered in groupoids over $\Sch/S$, one defines $\cX(\fU)$ to be the category where the objects are a sequence of arrows $ a_0 \arr a_1 \arr \cdots $ over the nilpotent thickenings $U_0 \hookarr U_1 \hookarr  \cdots$ and a morphism $(a_0 \arr a_1 \arr \cdots) \arr (a_0' \arr a_1' \arr \cdots)$ is a compatible sequence of arrows $a_i \arr a'_i$ over the identity.  One checks that if $\fI$ is replaced with a different ideal of definition, then one obtains an equivalent category.  Given a morphism of formal schemes $p: \fU' \arr \fU$, one obtains a functor $p^*: \cX(\fU) \arr \cX(\fU')$.  

If $G \arr S$ is a group scheme over $S$ with multiplication $\mu$ acting on the formal scheme $\fU$ via $\sigma: G \times_S \fU \arr \fU$ such that $\fI$ is an invariant ideal of definition, we say that  \emph{$G$ acts on a deformation $\hat a = (a_0 \arr a_1 \cdots) \in \cX(\fU)$}, if as above there is an arrow $\phi: \sigma^* \hat a \arr p_2^* \hat a$ in $\cX(G \times_S \fU)$ satisfying the cocycle $p_{23}^* \phi \, \circ \, (\id \times \sigma)^* \phi = (\mu \times \id)^* \phi$.  This is equivalent to giving compatible morphisms $[U_i / G] \arr \cX$.  (Given an appropriate definition of a formal stack $[\fU / G]$, this should be equivalent to giving a morphism $[\fU/G] \arr \cX$.)

\section{Local quotient structure}

We show that for closed points with linearly reductive stabilizer, the stabilizer acts algebraically on the formal deformation space.  In other words, Artin stacks are ``formally locally'' quotient stacks around such points.  This gives a formally local answer to Conjecture \ref{local_quotient_conj}.  We will use the same method as in \cite{tame} to deduce that all nilpotent thickenings are quotient stacks.   

\subsection{Deformation theory of $G$-torsors}

We will need to know the deformation theory of $G$-torsors over Artin stacks.  We recall for the reader the necessary results of the deformation theory of $G$-torsors from \cite{olsson_deformation} and \cite{tame}.

Suppose $G \arr S$ is a fppf group scheme and $p: \cP \arr \cX$ is a $G$-torsor.  Let $i: \cX \arr \cX'$ be a closed immersion of stacks defined by a square-zero ideal $I \subseteq \oh_{\cX'}$.  Then the collection of $2$-cartesian diagrams
$$\xymatrix{
\cP \ar[d]^p \ar@{-->}[r]^{i'} 		& \cP' \ar@{-->}[d]^{p'} \\
\cX \ar[r]^{i}					& \cX'
}$$
with $p': \cP' \arr \cX'$ a $G$-torsor form in a natural way a category.

\begin{prop}  \label{def_torsor}
Let $L_{BG/S}$ denote the cotangent complex of $BG \arr S$ and $f: \cX \arr BG$ be the morphism corresponding to the $G$-torsor $p:\cP \arr \cX$.
\begin{enumeratei}
\item There is a canonical class $o(x,i) \in Ext^1(Lf^* L_{BG/S}, I)$ whose vanishing is necessary and sufficient for the existence of an extension $(i',p')$ filling in the diagram
\item If $o(x,i) = 0$, then the set of isomorphism of extensions filling in the diagram is naturally a torsor under $Ext^0(Lf^* L_{BG/S}, I)$.
\item  For any extension $(i',p')$,  the group of automorphisms of $(i',p')$ (as a deformation of $\cP \arr \cX$) is canonically isomorphic to $Ext^{-1}(Lf^* L_{BG/S}, I)$.
\end{enumeratei} 
\end{prop}

\begin{proof} This is a special case of \cite[Theorem 1.5]{olsson_deformation} with $\cY = \cY' = BG$ and $Z = Z' = S$. \end{proof}

\begin{prop} \label{torsor_vanishing} Let $G \arr S$ be an fppf group scheme.  Then
\begin{enumeratei}
\item $L_{BG/S} \in D_{coh}^{[0,1]} (\oh_{BG})$.
\item If $G \arr S$ is smooth, $L_{BG/S} \in D_{coh}^{[1]} (\oh_{BG})$.
\end{enumeratei}
If $G \arr \Spec k$ is linearly reductive and $\cF$ is a coherent sheaf on $BG$, then
\begin{enumeratei} \setcounter{enumi}{2}
\item  $Ext^i(L_{BG/k}, \cF) = 0$ for $i \ne -1, 0$.
\item If $G \arr \Spec k$ is smooth, $Ext^i(L_{BG/k}, \cF) = 0$ for $i \ne -1$.
\end{enumeratei}
\end{prop}

\begin{proof} Part (i) and (ii) follow from the distinguished triangle induced by the composition $S \arr BG \arr S$ as in \cite[Lemma 2.18]{tame} with the observation that $G \arr S$ is a local complete intersection. 
Part (iii) is given in the proof of \cite[Lemma 2.17]{tame} and (iv) is clear from (ii).
\end{proof}

\subsection{Proof of Theorem 1}

\begin{proof} We prove inductively that each $\cX_i$ is a quotient stack by $G_x$ using deformation theory.  For (i), let $p_0: U_0 = \Spec k \arr \cX_0$ be the canonical $G_x$-torsor.  Suppose we have a compatible family of $G_x$-torsors $p_i: U_i \arr \cX_i$ with $U_i$ affine.  This gives a 2-cartesian diagram
$$\xymatrix{
U_0 \ar[r] \ar[d]^{p_0} 	& \ldots \ar[r]  & U_{n-1} \ar@{-->}[r]^{j_n} \ar[d]^{p_{n-1}}		& U_n  \ar@{-->}[d]^{p_n}		 \\
\cX_0	\ar[r]	& \ldots \ar[r]  & \cX_{n-1} \ar[r]^{i_n}	& \cX_n 
}$$
By Corollary \ref{def_torsor}, the obstruction to the existence of a $G_x$-torsor $p_n: U_n \arr \cX_n$ restricting to $p_{n-1}: U_{n-1} \arr \cX_{n-1}$ is an element 
$$o \in \Ext^1(Lf^*L_{BG_x/k}, \cI^n)= \Ext^1(L_{BG_x/k}, \cI^n / \cI^{n+1}) = 0$$
where $f: \cX_{n-1} \arr BG_x$ is the morphism defined by $U_{n-1} \arr \cX_{n-1}$ and $\cI$ denotes the sheaf of ideals defining $\cX_0$.  The vanishing is implied by Proposition \ref{torsor_vanishing}(iii).  Therefore, there exists a $G_x$-torsor $U_n \arr \cX_n$ extending $U_{n-1} \arr \cX_{n-1}$.  Since $U_0$ is affine, so is $U_n$ and the $G_x$-torsor $p_n$ gives an isomorphism $\cX_n  \cong [U_n / G_x]$. Furthermore, if $G_x$ is smooth, this extension is unique by Proposition \ref{torsor_vanishing}(iv). 

For (ii), first choose a scheme $S'$ and an \'etale morphism $S' \arr S$ such that $S' \times_S k(s) = k$.  Let $s' \in S'$ denote the preimage of $s$ and $S'_n  = \Spec \oh_{S',s'}/ \fm_{s'}^{n+1}$.  The group scheme $G_0 = G_x \arr \Spec k$ extends uniquely to smooth affine group schemes $G_i \arr S'_n$ (\cite[Expose III, Thm. 3.5]{sga3}) which by \cite[Prop. 3.9(iii)]{alper_good_arxiv} are linearly reductive.  If $\cX' = \cX \times_S S'$, then $BG_x \hookarr \cX'$ is a closed immersion with nilpotent thickenings $\cX'_n$ isomorphic to $\cX_n \times_S S'$.  Let $p_0: U_0 = \Spec k \arr \cX_0$ be the canonical $G_x$-torsor which we may also view as a torsor over $G_n \arr S'_n$.  Suppose we have a compatible family of $G_n \arr S_n$-torsors $p_i: U_i \arr \cX_i$ with $U_i$ affine.  This gives a 2-cartesian diagram
$$\xymatrix{
U_0 \ar[r] \ar[d]^{p_0} 	& \ldots \ar[r]  & U_{n-1} \ar@{-->}[r]^{j_n} \ar[d]^{p_{n-1}}		& U_n  \ar@{-->}[d]^{p_n}		 \\
\cX'_0	\ar[r]	& \ldots \ar[r]  & \cX'_{n-1} \ar[r]^{i_n}	& \cX'_n 
}$$
of Artin stacks over $S'_n$.  By Corollary \ref{def_torsor}, the obstruction to the existence of a $G_x$-torsor $p_n: U_n \arr \cX'_n$ restricting to $p_{n-1}: U_{n-1} \arr \cX'_{n-1}$ is an element 
$$o \in \Ext^1(Lf^*L_{BG_n/S_n}, \cI^n)= H^2(BG_x, \mathfrak{g} \tensor \cI^n / \cI^{n+1}) = 0$$
where $f: \cX'_{n-1} \arr BG_x$ is the morphism defined by $U_{n-1} \arr \cX'_{n-1}$.  Since the set of extensions is $H^1(BG_x, \mathfrak{g} \tensor \cI^n / \cI^{n+1}) = 0$, there is a unique extension $p_n: U_n \arr \cX'_n$.
\end{proof}

\begin{cor} Let $\cX$ be a locally noetherian Artin stack over $\Spec k$ and $\xi \in |\cX|$ be a closed point with smooth, affine and linearly reductive stabilizer.  Let  $x: \Spec k \arr \cX$ be a representative of $\xi$. Then there exists a miniversal deformation $(A, \hat{\xi})$ of $x$ with $G_x$-action, which is unique up to $G_x$-invariant isomorphism.  
\end{cor}


\begin{proof}  The first statement follows directly from the above theorem with the observation that $\dlim U_i \arr \cX$ is a miniversal deformation. \end{proof}

\begin{remark}  The action of $G_x$ on $\Spf A$ fixes the maximal ideal so we get an induced \emph{algebraic} action of $G_x$ on $(\fm / \fm)^{\dual}$.  The miniversality of $\xi$ gives an identification of $k$-vector spaces  $\Psi: (\fm / \fm^2)^{\dual} \iso \bar F_x(k[\epsilon])$ which we claim is $G_x$-equivariant.

The map $\Psi$ is defined as follows:  if  $\tau: \Spec k[\epsilon] \arr \Spec A / \fm^2$ there is an induced diagram
$$\vcenter{\xymatrix{
x \ar[r] \ar[rd]	&	\tau^* \xi_1 \ar[d] 	\ar[d]\\
			& \xi_1	
}}
\qquad \text{ over } \qquad
\vcenter{\xymatrix{
 \Spec k \ar[r] \ar[rd]	& \Spec k[\epsilon] \ar[d]\\
				& \Spec A / \fm^2
}}$$
then $\Psi(\tau) = (x \arr \tau^* \xi_1)$.  The action of $G_x$ on $\bar F_x(k[\epsilon])$ is given in Section \ref{inf_action}.  Under the identification $(\fm / \fm^2)^{\dual} \cong \bar F_{U, u}(k[\epsilon])$ where $U = \Spec A/\fm^2$ and $u: \Spec k \arr U$ is the closed point, then $G_x$-action on $\bar F_{U, u}(k[\epsilon])$ can be given explicitly:  If $p: \Spec B \arr \Spec k$, then an element $g \in G_x(R)$  gives a $B$-algebra isomorphism $\alpha_g: A/\fm^2 \tensor_k B \arr A/\fm^2 \tensor_k B$ and an element $\sigma \in \bar F_{U, u}(k[\epsilon])$ corresponds to a $B$-module homomorphism $A/\fm^2 \tensor_k B$ and $g \cdot \sigma \in  \bar F_{U, u}(k[\epsilon])$ is the $B$-module homomorphism corresponding to the composition $A/\fm^2 \tensor_k B \stackrel{\alpha_{g}^{-1}}{\arr} A/\fm^2 \tensor_k B \arr B$.

We also note that if $p: \Spec B \arr \Spec k$, then under the isomorphisms given in Section \ref{inf_action}, we have a commutative diagram
$$\xymatrix{
 \bar F_{U, u}(k[\epsilon]) \tensor_k B \ar[r]^{\sim} \ar[d]^{\Psi \tensor_k B}	& \bar F_{U, p^*u}(B[\epsilon]) \ar[d]^{\Psi_B} \\
 \bar F_{\cX, x}(k[\epsilon]) \tensor_k B \ar[r]^{\sim}	& \bar F_{\cX, p^*x}(B[\epsilon]) \\
 }$$
where for $(\tau: \Spec B[\epsilon] \arr U) \in \bar F_{U, p^*u}(B[\epsilon]) $, $\Psi_B(\tau) = (p^*x \arr \tau^* \xi_1)$.

For $g \in G_x(B)$ and $(\tau: \Spec B[\epsilon] \arr U) \in \bar F_{U, u}(B[\epsilon])$, the pullback of the cocycle $\phi: \sigma^* \xi_1 \arr p_2^* \sigma$ (defining the $G_x$-action on $\xi_1$) under the morphism $(g,\id): \Spec B \times_k U \arr G_x \times_k U$ gives an arrow $\beta$ making a commutative diagram
$$ \vcenter{\xymatrix{
p^*x \ar[r]^{g} \ar[d]	& p^*x \ar[d] \\
p_2^* \xi_1 \ar[r]^{\beta}	& p_2^* \xi_1
}}
\qquad \text{ over } \qquad
\vcenter{\xymatrix{
\Spec B \ar[r]^{=} \ar[d]	& \Spec B \ar[d] \\
\Spec B \times_k U \ar[r]^{\alpha_g}	& \Spec B \times_k U
}}$$
We have a commutative diagram
$$\vcenter{\xymatrix{
p^*x \ar[r]^{g} \ar[rd] \ar[rdd]	& p^*x \ar[rd] \ar[rdd] \\
					& \tau^* \xi_1 \ar[r]^{\gamma} \ar[d]		& (\alpha_g \circ \tau)^* \xi_1 \ar[d]	\\
				& p_2^* \xi_1 \ar[r]^{\beta}			& p_2^* \xi_1	
}} \qquad \text{ over } \qquad
\vcenter{\xymatrix{
\Spec B \ar[r]^{=} \ar[rd] \ar[rdd]	& \Spec B \ar[rd] \ar[rdd]\\
	& \Spec B[\epsilon] \ar[r]^{=} \ar[d]^{\tau}		& \Spec B[\epsilon] \ar[d]^{\alpha_g \circ \tau}\\
	& \Spec B \times_k U \ar[r]^{\alpha_g}				& \Spec B \times_k U
}}
$$
where $\gamma: \tau^* \xi_1 \arr (g \circ \tau)^* \xi_1$ is the unique arrow making the bottom square commute.  The arrow $\gamma$ identifies $g \cdot \Psi(\tau) = (p^*x \stackrel{\alpha_{g}^{-1}}{\arr} p^*x \arr \tau^* \xi_1)$ and $\Psi(g \cdot \tau) = (x \arr (\alpha_g \circ \tau)^* \xi_1)$. 
\end{remark}

\section{Luna's \'etale slice theorem} \label{luna_sect}

In this section, we recover Luna's \'etale slice theorem.  Many of the ingredients of the proof are stacky versions of Luna's methods.  However, we believe that using stacks allows for a more streamlined proof.  In \cite{luna}, it was necessary to prove and apply a $G$-equivariant version of Zariski's Main Theorem; we simply apply Zariski's Main Theorem for Artin stacks.  We remark that the method to prove \'etaleness of the induced map on quotients is different.  We apply a general result which gives sufficient conditions for an \'etale morphism of Artin stacks to induce an \'etale morphism on good moduli spaces (Theorem \ref{etale_preserving}) while Luna reduces to the case where $x \in X$ is normal so that  \'etaleness of the map between the quotients is equivalent to the morphism being unramified and injective on stalks, both of which can be checked algebraically.  We therefore have no normality assumptions in the fundamental lemma (Theorem \ref{fund_lem}).

\subsection{Equivariant linearizations}
A group action on a scheme $X$ affine and smooth over $S$ can be $S_X$-linearized if $S_X \arr X$ is linearly reductive:

\begin{lem} \label{equivariant_linearization}
Let $S$ be an affine scheme and $G \arr S$ be an fppf affine group scheme acting on a scheme $p: X \arr S$ with $p$ affine.  Let $f: T \arr X$ and suppose that the stabilizer $G_f \arr T$ is linearly reductive and $p$ is smooth at points in $f(T)$.  The stabilizer $G_f$ acts naturally on the $T$-schemes $X \times_S T$ and the pullback of the tangent bundle $T_{X/S} \times_X T$.  There exists a non-canonical $G_f$-equivariant morphism 
$$\xymatrix{
X \times_S T\ar[rd]^{p_{2}} \ar[r]		&	 T_{X/S} \times_X T \ar[d] \\
							& T
}$$
\end{lem}

\begin{proof}  The $G_f$ action on $X \times_S T = \sSpec f^* (p_2)_* \oh_{X \times_S X}$ and $T_{X/S} \times_X T = \sSpec \Sym^* f^* \Omega_{X/S}$ is induced from $G_f$-actions on the $\oh_T$-modules $f^* (p_2)_* \oh_{X \times_S X}$ and $f^* \Omega_{X/S}$.    To give an $G_f$-equivariant $T$-morphism $X \times_S T \arr T_{X/S} \times_X T$ it suffices to give a $G_f$-equivariant morphism of $G_f$-$\oh_T$-modules $f^* \Omega_{X/S} \arr f^* (p_2)_* \oh_{X \times_S X}$.  Let $\cI$ be the sheaf of ideals in $\oh_{X \times_S X}$ defining $\Delta: X \hookarr X \times_S X$.   There is a surjection $(p_2)_* \cI \arr \cI/\cI^2$ inducing an exact sequence
$$0 \arr \cK \arr f^* (p_2)_*  \cI \arr f^*(\cI/ \cI^2) \arr 0$$
of $G_f$-$\oh_T$-modules.  We may consider any $G_f$-$\oh_T$-module as an $\oh_{BG_f}$-module.  By observing that $f^* \Omega_{X/S}  \cong f^*(\cI/\cI^2)$ is locally free and applying $\Hom_{BG_f}(f^*(\cI/\cI^2), \cdot)$ 
$$\Ext^1_{BG_f}(f^*(\cI/\cI^2), \cK) = H^1(BG_f,  \cK \tensor f^*(\cI/\cI^2)^\dual) = 0$$
the sequence above splits.  Therefore there is an $G_f$-equivariant morphism $f^*(\cI/\cI^2) \arr f^*(p_2)_* \cI \arr f^* (p_2)_*\oh_{X \times_S X}$. \end{proof}

\begin{remark}  In general (with $f: T \arr X$ and $X \arr S$ affine and smooth at $f(T)$), there exists non-canonically an $T$-morphism $X \times_S T \arr T_{X/S} \times_X T$.  The hypothesis that $G_f \arr T$ is linearly reductive guarantees that this morphism can be constructed $G_f$-equivariantly.
\end{remark}

By applying the lemma with  $T = X$ and $f = \id$, we see that if the stabilizer $S_X \arr X$ is linearly reductive, then there is an $S_X$-invariant $X$-morphism $\Psi: X \times_S X \arr T_{X/S}$.  Suppose that $S = \Spec k$ and $x: \Spec k \arr X$.  Then the base change of the $X$-morphism $\Psi$ by $x: \Spec k \arr X$ yields a  $G_x$-invariant morphism $X \arr T_x$.  In this case, a smoothness hypothesis is not necessary to find a $G_x$-equivariantly linearization around point with linearly reductive stabilizer.  Of course, the induced morphism $X \arr T_x$ is \'etale only when $x \in X$ is smooth.

\begin{lem}  (\cite[Lemma on p. 96]{luna}) Suppose $G$ is an affine group scheme of finite type over a field $k$ acting on an affine scheme $X$ over $k$.  If $x \in X(k)$ is closed point with linearly reductive stabilizer $G_x$, then there is a linear action of $G_x$ on the tangent space $T_x$ and a $G_x$-equivariant morphism $X \arr T_x$ sending $x$ to the origin and inducing an isomorphism on tangent spaces.
\end{lem}

\begin{proof} Let $m \subseteq A$ be the maximal ideal of $x$.  Since $x$ is a fixed point under the induced action by $G_x$ on $\Spec A$, there is a dual action of $G_x$ on the $k$-vector space $m$ and a $G_x$-invariant map $m \arr m/m^2$.  There exists a finite dimensional $G_x$-invariant subspace $V' \subseteq m$ such $V' \twoheadrightarrow m/m^2$.  Since $G_x$ is linearly reductive, there is a $G_x$-invariant subspace $V \subseteq m$ with $V \iso m/m^2$.  This gives a homomorphism of rings
$$ \Sym^* m/m^2 \iso \Sym^*V \arr A $$
which induces the desired $G_x$-invariant morphism $X \arr T_x X$. \end{proof}

\begin{example} There are group actions on affine space that are not linear.  For instance, consider $\ZZ_2$ acting on $\AA^2$ by the involution $x \mapsto -x, y \mapsto -y + x^2$ with $x$ the origin. The $k$-vector space $\langle x,y \rangle \subseteq k[x,y]$ is not $G_x$-invariant but it is contained in the $G_x$-invariant $k$-vector space $\langle x,y,x^2 \rangle$.  This contains a $G_x$-invariant subspace $\langle x,y-\frac{1}{2}x^2 \rangle$ which maps $G_x$-equivariantly onto $m/m^2$.
\end{example}

\subsection{Descent of \'etaleness to good moduli spaces}

We begin by recalling a generalization of \cite[Lemma 1 on p.90]{luna} which gives sufficient criteria for when an \'etale morphism of Artin stacks induces an \'etale morphism of good moduli spaces.

\begin{thm} (\cite[Theorem 5.1]{alper_good_arxiv}) \label{etale_preserving}
Consider a commutative diagram $$\xymatrix{ 
\cX \ar[r]^f \ar[d]^{\phi}		& \cX' \ar[d]^{\phi'} \\
Y \ar[r]^g					& Y'
}$$
where $\cX, \cX'$ are locally noetherian Artin stacks, $g$ is locally of finite type, $\phi, \phi'$ are good moduli spaces and $f$ is representable.  Let $\xi \in |\cX|$.  Suppose

\begin{enumeratea}
\item  $f$ is \'etale at $\xi$.
\item $f$ is stabilizer preserving at $\xi$.
\item $\xi$ and $f(\xi)$ are closed.
\end{enumeratea}
Then $g$ is \'etale at $\phi(\xi)$.
\end{thm}

\begin{cor} \label{etale_preserving_cor}
Consider a commutative diagram $$\xymatrix{ 
\cX \ar[r]^f \ar[d]^{\phi}		& \cX' \ar[d]^{\phi'} \\
Y \ar[r]^g					& Y'
}$$
with $\cX, \cX'$ locally noetherian Artin stacks of finite type over $S$, $g$ locally of finite type, and $\phi, \phi'$ good moduli spaces.  If $f$ is \'etale, pointwise stabilizer preserving and weakly saturated, then $g$ is \'etale.
\end{cor}

\begin{proof}
It suffices to check that $g$ is \'etale at closed points $y \in Y$.  There exists a unique closed point $\xi \in|\cX|$ above a closed point $y \in |Y|$.  The image $s \in S$ is locally closed and we may assume it is closed.  Since $f$ is weakly saturated, by base changing by $\Spec k(s) \arr S$, we have that $\cX_s \arr \cX'_s$ maps closed points to closed points so that $f(\xi) \in |\cX'_s|$ is closed and therefore $f(\xi) \in |\cX'|$ is closed.  It follows from the above theorem that $g$ is \'etale at $y$.  
\end{proof}

We will need the following generalization of \cite[Lemma p.89]{luna}.  Note that here we replace the hypothesis in \cite[Proposition 6.4]{alper_good_arxiv} that $f$ maps closed points to closed points with the  hypothesis that $f$ is weakly saturated.

\begin{prop}   \label{finite_prop}
Suppose $\cX, \cX'$ are locally noetherian Artin stacks and
$$\xymatrix{
\cX \ar[r]^f \ar[d]^{\phi'}		& \cX' \ar[d]^{\phi} \\
Y \ar[r]^{g}					& Y'
}$$
is commutative with $\phi,\phi'$ good moduli spaces.  Suppose
\begin{enumeratea}
\item $f$ is representable, quasi-finite and separated.
\item $g$ is finite
\item $f$ is weakly saturated.
\end{enumeratea}
Then $f$ is finite. 
\end{prop}

\begin{proof}
We may assume $S$ and $Y'$ are affine schemes.  Furthermore, $\cX \arr Y \times_{Y'} \cX'$ is representable, quasi-finite, separated and weakly saturated so we may assume that $g$ is an isomorphism.  By Zariski's Main Theorem (\cite[Thm. 16.5]{lmb}), there exists a factorization
$$\xymatrix{
\cX \ar[r]^I \ar[rd]^f		& \cZ \ar[d]^{f'}\\
					& \cX'
}$$
where $I$ is a open immersion, $f'$ is a finite morphism and $\oh_{\cZ} \hookrightarrow I_* \oh_{\cX}$ is an inclusion.  Since $\cX'$ is cohomologically affine and $f'$ is finite, $\cZ$ is cohomologically affine and admits a good moduli space $\varphi: \cZ \arr Y$.  

Since $f$ is weakly saturated, $I$ is weakly saturated.  Since $\cX$ and $\cZ$ admit the same good moduli space, by Remark \ref{saturated_open}, $I$ must be an isomorphism.  
\end{proof}


\begin{prop} \label{square_prop}
Suppose $\cX, \cX'$ are locally noetherian Artin stacks and 
$$\xymatrix{
\cX \ar[r]^f \ar[d]^{\phi'}		& \cX' \ar[d]^{\phi} \\
Y \ar[r]^{g}					& Y'
}$$
is a commutative diagram with $\phi,\phi'$ good moduli spaces.  If $f$ is representable, separated, \'etale, stabilizer preserving and weakly saturated, then $g$ is \'etale and the diagram is cartesian.
\end{prop}

\begin{proof}
Since $f$ is representable, separated, quasi-finite, stabilizer preserving and weakly saturated, so is $\Psi: \cX \arr \cX' \times_{Y'} Y$.  By Proposition \ref{finite_prop}, $\Psi$ is finite.  Moreover, by Corollary \ref{etale_preserving_cor}, $g$ is \'etale, and therefore so is $\Psi$.   Therefore, $\Psi: \cX \to \cX' \times_{Y'} Y$ is a finite, \'etale morphism between Artin stacks which both have $Y$ has a good moduli space but since $\Psi$ is also stabilizer preserving, it follows that $\Psi$ has degree $1$ and is therefore an isomorphism.
\end{proof}

\begin{remark}  The conditions above that $f$ is stabilizer preserving and weakly saturated are necessary (even if one requires further that $g$ is \'etale).  Indeed, both the open immersion $\Spec k \to [\AA^1 / \GG_m]$ and the \'etale presentation $\Spec k \to B G$ for a finite group $G$ induce isomorphisms on good moduli spaces but the corresponding diagrams are not cartesian.
\end{remark}

\subsection{The fundamental lemma}
The fundamental lemma expands on Theorem \ref{etale_preserving} by guaranteeing that after shrinking Zariski-locally on the good moduli spaces, one has \'etaleness everywhere and that the induced square is even cartesian.  Although we will only use the lemma in the case when $\cX, \cX'$ are quotient stacks, we would like to stress precisely where the quotient stack structures are used in Luna's slice theorem as to emphasize the difficulties at proving Conjecture \ref{local_quotient_conj} in general.

\begin{thm}  \label{fund_lem} Consider a commutative diagram 
$$\xymatrix{ 
\cX \ar[r]^f \ar[d]^{\phi}		& \cX' \ar[d]^{\phi'} \\
Y \ar[r]^g					& Y'
}$$
with $\cX, \cX'$ locally noetherian Artin stacks over a scheme $S$ and $\phi, \phi'$ good moduli spaces with $Y$ and $Y'$ algebraic spaces.   Suppose $f$ is representable and separated, and both $f$ and $g$ are locally of finite type.  Suppose:

\begin{enumeratea}
\item $f$ is stabilizer preserving at $\xi$
\item  $f$ is \'etale at $\xi$.
\item $\xi$ and $f(\xi)$ are closed.
\end{enumeratea}
Then there exist a Zariski sub-algebraic space  $Y_1 \subseteq Y$ such that $f|_{\phi^{-1}(Y_1)}$ is \'etale, $g|_{Y_1}$ is \'etale, and the diagram
\begin{equation} \label{cartesian-fundamental}
\xymatrix{ 
\phi^{-1}(Y_1)\ar[r]^f \ar[d]^{\phi}		& \cX' \ar[d]^{\phi'} \\
Y_1 \ar[r]^g					& Y'
}
\end{equation}
is cartesian. 
 If $Y$ and $Y'$ are schemes such that $Y'$ has affine diagonal, then $Y_1$ and $g(Y_1)$ can be chosen to be affine.
\end{thm}

\begin{proof}
The hypotheses imply by Theorem \ref{etale_preserving} that $g$ is \'etale at $\phi(\xi)$.  The closed subset of $|\cX|$ 
$$Z = \{\eta \in |\cX| \big| f \textrm{ is not \'etale at } \eta \textrm{ or }  g \textrm{ is not \'etale at } \phi( \eta) \}$$
is disjoint from the closed subset $\{\xi\}$.  Therefore, $U=Y \setminus \phi(Z)$ is an open sub-algebraic space containing $\phi(\xi)$ such that $f|_{\phi^{-1}(U)}$ and $g|_{U}$ are \'etale.  So we may assume that $f$ and $g$ are \'etale.

We now shrink further to ensure that $f$ is weakly saturated.  By Zariski's Main Theorem (\cite[Thm. 16.5]{lmb}), there exists a factorization $f: \cX \hookarr \cW \to \cX'$ such that $\cX \hookarr \cW$ is an open immersion and $\cW \to \cX'$ is finite.  Since $\cW \to \cX$ is finite, there exists a good moduli space $\varphi: \cW \to W$.  Since $f(\xi) \in |\cX'|$ is closed, $\xi \in |\cW|$ is closed.  Therefore, $\{\xi\}$ and $\cW \setminus \cX$ are disjoint closed substacks of $\cW$, which implies that $\xi \notin \cZ:=\varphi^{-1}(\varphi(\cW \setminus \cX))$.  Set $U = Y \setminus \phi(\cX \cap \cZ)$.  Then $\phi^{-1}(U) = \cW \setminus \cZ$ which is a saturated open substack of $\cW$.  Since finite morphisms are weakly saturated, the composition $\phi^{-1}(U) \to \cW \to \cX'$ is weakly saturated.
Therefore, we have a commutative diagram
$$\xymatrix{ 
\phi^{-1}(U) \ar[r]^f \ar[d]^{\phi}		& \cX' \ar[d]^{\phi'} \\
U \ar[r]^g					& Y'
}$$
where $f$ is representable, separated, \'etale, and weakly saturated and $g$ is \'etale.  Now $\Psi: \phi^{-1}(U) \to U \times_{Y'} \cX'$ is a representable, separated and quasi-finite morphism of Artin stacks both having $U$ as a good moduli space.  By Proposition \ref{finite_prop}, $\Psi$ is finite.  Since $f$ is stabilizer preserving at $\xi$, so is $\Psi$; it follows that $\Psi$ is a finite \'etale morphism of degree $1$ at $\xi$.  The open substack $\cU' \subseteq \phi^{-1}(U)$ where $\Psi$ is an isomorphism contains $\xi$.  By setting $Y_1 = U \setminus \phi( \phi^{-1}(U) \setminus \cU')$,  we obtain a diagram as in (\ref{cartesian-fundamental}) which is cartesian.

For the final statement, if $Y$ is a scheme, then in the above argument we can choose $Y_1$ to be affine.  If $Y'$ is a scheme with affine diagonal, then we can choose $Y'_1$ to be an open affine subscheme of $g(Y_1)$ and it follows that $Y_1 \cap g^{-1}(Y'_1)$ is also affine.
\end{proof}

\subsection{Proof of Theorem \ref{luna_thm}}

\begin{proof}  For (i), since the orbit $o(f) \hookarr X$ is closed, by \cite[Theorem 12.14]{alper_good_arxiv}, the stabilizer $G_f = S_X \times_X S$ is linearly reductive over $S$.  Lemma \ref{equivariant_linearization} gives an $G_f$-invariant $S$-morphism 
$$g: X \arr T_{X/S} \times_X S$$
For $s \in S$, the induced morphism on fibers is $g_s: X_s \arr T_{X_s,f(s)}$ which induces an isomorphism on tangent spaces at $f(s)$.  Since $X_s \arr \Spec k(s)$ is smooth, $g_s$ is \'etale at $f(s)$.  Since $X \arr S$ is flat at points in $f(S)$, by fibral flatness (\cite[IV.11.3.10]{ega}), $g$ is flat at points in $f(S)$.  Since the property of being unramified can be checked on fibers, $g$ is \'etale at points in $f(S)$.  
Furthermore, $o(f) \arr S$ is smooth and there is an $G_f$-equivariant inclusion $T_{o(f)/S} \times_X S\subseteq T_{X/S} \times_X S$.  Since $G_f \arr S$ is linearly reductive and $X \arr S$ is smooth over $f(S)$, there is a decomposition of $\oh_{BG_f}$-modules
$$f^* \Omega_{X/S} \cong f^* \Omega_{o(f)/S} \oplus \cF$$
and $N = \sSpec_S \Sym^* \cF \subseteq T_{X/S} \times_X S$ is a space normal to the tangent space of the orbit $T_{o(f)/S} \times_X S$ inheriting a $G_f$-action.  If we define $W = g^{-1}(N)$, then $W$ is a $G_f$-invariant closed subscheme of $X$ and $f: S \arr X$ factors as a composition $w: S \arr W$ and $W \hookarr X$. If $\cW = [W/G_f]$ and $\cN = [N/G_f]$, the induced maps
$$\xymatrix{
		& \cW \ar[rd] \ar[ld] \\
\cN		&				& \cX
}$$
are \'etale at $w(S)$.  By applying the Fundamental Lemma \ref{fund_lem}, we have established (i).  

For (ii), let $X \hookarr X'$ be a $G$-equivariant embedding into a smooth affine $S$-scheme $X'$.  If $W' \subseteq X'$ satisfies the conditions of the theorem with $\psi': \cW' \arr V'$ a good moduli space.  Then $\cW = \cW' \times_{\cX'} \cX = [W/G_f]$ for an open affine $G_f$-invariant open $W \subseteq X$.  We have a commutative cube
$${\def\objectstyle{\scriptstyle}
\def\labelstyle{\scriptstyle}
\xymatrix@=20pt{
				&\cW \ar[rr] \ar[dd] \ar[dl]			&				& \cX \ar[dd] \ar[dl] \\
\cW' \ar[rr] \ar[dd]&						& \cX' \ar[dd]	& \\
				&V \ar[rr] \ar[dl]				& 				& Y \ar[dl] \\
V' \ar[rr]	&							&Y'		&
}}$$
where the vertical arrows are good moduli spaces and the arrows out of the page are closed immersions.  The top square and front square are cartesian.  We claim that the bottom square is also cartesian.  Indeed, there is a commutative diagram
$$\xymatrix{
\cW \ar[r] \ar[d]				& \cX \ar[d] \\
V' \times_{Y'} Y \ar[r] \ar[d]		& Y \ar[d] \\
V' \ar[r]					& Y'
}$$
Since the big square and the bottom square are cartesian, the top square is cartesian.  Therefore, $\cW \arr V' \times_{Y'} Y$ is a good moduli space so by uniqueness, the induced map $V \arr V' \times_{Y'} Y$ is an isomorphism.  Therefore, in the cube, the back square is cartesian and the horizontal arrows are \'etale. \end{proof}


\bibliography{../references}{}

\def\polhk#1{\setbox0=\hbox{#1}{\ooalign{\hidewidth
  \lower1.5ex\hbox{`}\hidewidth\crcr\unhbox0}}} \def\cprime{$'$}
\providecommand{\bysame}{\leavevmode\hbox to3em{\hrulefill}\thinspace}
\providecommand{\MR}{\relax\ifhmode\unskip\space\fi MR }
\providecommand{\MRhref}[2]{%
  \href{http://www.ams.org/mathscinet-getitem?mr=#1}{#2}
}
\providecommand{\href}[2]{#2}
\begin{thebibliography}{AOV08}

\bibitem[Alp08]{alper_good_arxiv}
Jarod Alper, \emph{Good moduli spaces for artin stacks}, math.AG/0804.2242
  (2008).

\bibitem[AOV08]{tame}
Dan Abramovich, Martin Olsson, and Angelo Vistoli, \emph{Tame stacks in
  positive characteristic}, Ann. Inst. Fourier (Grenoble) \textbf{58} (2008),
  no.~4, 1057--1091.

\bibitem[Art74]{artin_versal}
Michael Artin, \emph{Versal deformations and algebraic stacks}, Invent. Math.
  \textbf{27} (1974), 165--189.

\bibitem[AV02]{abramovich-vistoli}
Dan Abramovich and Angelo Vistoli, \emph{Compactifying the space of stable
  maps}, J. Amer. Math. Soc. \textbf{15} (2002), no.~1, 27--75 (electronic).

\bibitem[Con05]{conrad}
Brian Conrad, \emph{Keel-mori theorem via stacks},
  \url{http://www.math.stanford.edu/~bdconrad/papers/coarsespace.pdf} (2005).

\bibitem[Gro67]{ega}
Alexander Grothendieck, \emph{\'{E}l\'ements de g\'eom\'etrie alg\'ebrique},
  Inst. Hautes \'Etudes Sci. Publ. Math. (1961-1967),
  no.~4,8,11,17,20,24,28,32.

\bibitem[KM97]{keel-mori}
Se{\'a}n Keel and Shigefumi Mori, \emph{Quotients by groupoids}, Ann. of Math.
  (2) \textbf{145} (1997), no.~1, 193--213.

\bibitem[Knu71]{knutson}
Donald Knutson, \emph{Algebraic spaces}, Springer-Verlag, Berlin, 1971, Lecture
  Notes in Mathematics, Vol. 203.

\bibitem[Kol97]{kollar_quotients}
J{\'a}nos Koll{\'a}r, \emph{Quotient spaces modulo algebraic groups}, Ann. of
  Math. (2) \textbf{145} (1997), no.~1, 33--79.

\bibitem[LMB00]{lmb}
G{\'e}rard Laumon and Laurent Moret-Bailly, \emph{Champs alg\'ebriques},
  Ergebnisse der Mathematik und ihrer Grenzgebiete. 3. Folge. A Series of
  Modern Surveys in Mathematics [Results in Mathematics and Related Areas. 3rd
  Series. A Series of Modern Surveys in Mathematics], vol.~39, Springer-Verlag,
  Berlin, 2000.

\bibitem[Lun73]{luna}
Domingo Luna, \emph{Slices \'etales}, Sur les groupes alg\'ebriques, Soc. Math.
  France, Paris, 1973, pp.~81--105. Bull. Soc. Math. France, Paris, M\'emoire
  33.

\bibitem[Ols06a]{olsson_deformation}
Martin Olsson, \emph{Deformation theory of representable morphisms of algebraic
  stacks}, Math. Z. \textbf{253} (2006), no.~1, 25--62.

\bibitem[Ols06b]{olsson_homstacks}
\bysame, \emph{{$\underline {\rm Hom}$}-stacks and restriction of scalars},
  Duke Math. J. \textbf{134} (2006), no.~1, 139--164.

\bibitem[Ols07]{olsson_crystalline}
\bysame, \emph{Crystalline cohomology of algebraic stacks and {H}yodo-{K}ato
  cohomology}, Ast\'erisque (2007), no.~316, 412 pp. (2008).

\bibitem[Pin74]{pinkham}
Henry Pinkham, \emph{Deformations of algebraic varieties with {$G\sb{m}$}
  action}, Soci\'et\'e Math\'ematique de France, Paris, 1974, Ast{\'e}risque,
  No. 20.

\bibitem[Rim80]{rim_equivariant}
Dock Rim, \emph{Equivariant {$G$}-structure on versal deformations}, Trans.
  Amer. Math. Soc. \textbf{257} (1980), no.~1, 217--226.

\bibitem[Ryd13]{rydh_quotients}
David Rydh, \emph{Existence and properties of geometric quotients}, J.
  Algebraic Geom., to appear (2013).

\bibitem[Sch68]{schlessinger}
Michael Schlessinger, \emph{Functors of {A}rtin rings}, Trans. Amer. Math. Soc.
  \textbf{130} (1968), 208--222.

\bibitem[sga64]{sga3}
\emph{Sch\'emas en groupes}, S\'eminaire de G\'eom\'etrie Alg\'ebrique du Bois
  Marie 1962/64 (SGA 3). Dirig\'e par M. Demazure et A. Grothendieck. Lecture
  Notes in Mathematics, Vol. 151,152,153, Springer-Verlag, Berlin, 1962/1964.

\bibitem[Sum74]{sumihiro1}
Hideyasu Sumihiro, \emph{Equivariant completion}, J. Math. Kyoto Univ.
  \textbf{14} (1974), 1--28.

\bibitem[Tho87]{thomason}
R.~W. Thomason, \emph{Equivariant resolution, linearization, and {H}ilbert's
  fourteenth problem over arbitrary base schemes}, Adv. in Math. \textbf{65}
  (1987), no.~1, 16--34. \MR{MR893468 (88g:14045)}

\end{thebibliography}
\bibliographystyle{amsalpha}

\end{document}